\documentclass{amsart}

\usepackage{amsmath,amssymb,amsthm}
\usepackage[hidelinks]{hyperref}
\usepackage{graphicx}
\usepackage[noadjust]{cite}
\usepackage{xurl}
\usepackage[usenames, dvipsnames]{xcolor}
\usepackage[linesnumbered,ruled]{algorithm2e}
\usepackage{dsfont}
\usepackage{tikz}
\usepackage{mathtools}    
\usepackage{enumitem}
\usepackage{tikz}
\usetikzlibrary{trees}

\allowdisplaybreaks

\newtheorem{theorem}{Theorem}[section]
\newtheorem{lemma}[theorem]{Lemma}
\newtheorem{prop}[theorem]{Proposition}
\newtheorem{cor}[theorem]{Corollary}

\theoremstyle{definition}
\newtheorem{definition}[theorem]{Definition}

\newtheorem{obs}[theorem]{Observation}

\theoremstyle{remark}
\newtheorem{remark}[theorem]{Remark}

\numberwithin{equation}{section}

\def\cB{\mathcal{B}}

\def\cN{\mathcal{N}}

\def\cP{\mathcal{P}}
\def\cA{\mathcal{A}}

\def\eps{\varepsilon}
\def\N{\mathbb{N}}

\def\F{\mathbb{F}}

\def\b1{\mathds{1}}
\def\Z{\mathbb{Z}}
\def\Q{\mathbb{Q}}

\DeclareMathOperator\rank{rank}
\DeclareMathOperator\ord{ord}

\newcommand{\BB}{\mathcal{B}}

\newcommand{\CC}{\mathcal{C}}

\newcommand{\OO}{\mathcal{O}}
\newcommand{\pp}{\mathfrak{p}}

\newcommand{\res}{\operatorname{res}}
\newcommand{\Gal}{\operatorname{Gal}}
\newcommand{\id}{\operatorname{id}}

\renewcommand{\leq}{\leqslant}
\renewcommand{\geq}{\geqslant}

\newcommand{\gen}[1]{\langle #1 \rangle}

\begin{document}

\title{Deterministic polynomial factorisation modulo many primes}

\author{Daniel Altman}
\address{Department of Mathematics, Stanford University, CA 94305, USA}
\curraddr{}
\email{daniel.h.altman@gmail.com}
\thanks{}



\date{}

\dedicatory{}

\begin{abstract}
Designing a deterministic polynomial time algorithm
for factoring univariate polynomials over finite fields
remains a notorious open problem.
In this paper, we present an unconditional deterministic algorithm
that takes as input an irreducible polynomial $f \in \Z[x]$,
and computes the factorisation of its reductions modulo~$p$
for all primes $p$ up to a prescribed bound $N$.
The \emph{average running time per prime} is polynomial in the size
of the input and the degree of the splitting field of $f$ over $\Q$.
In particular, if $f$ is Galois,
we succeed in factoring in (amortised) deterministic polynomial time.
\end{abstract}

\maketitle

\section{Introduction}\label{s:intro}

A central problem in computational number theory
is to design efficient algorithms for factoring
univariate polynomials over finite fields.
If randomisation is permitted,
then factorisation algorithms running in polynomial time
go back to at least \cite{Ber70},
and many such algorithms are widely used in practice;
for a survey, see \cite[Ch.\,14]{vzGG13}, \cite{vzGP01} or \cite{Shp99}.
The current asymptotic complexity record is held by \cite{KU11}. A reader interested in randomised algorithms may also wish to consult \cite{CZ81}, \cite{KS98}, \cite{GNU16}.

For the rest of the paper we restrict our attention to
\emph{deterministic} algorithms.
To simplify the discussion,
we focus on the case of a polynomial over a prime field,
say $f \in \F_p[x]$ of degree $d \geq 2$, where $p$ is a prime number.
Here and throughout the paper, to factorise (or factor) $f$ means to find a complete factorisation of $f$
into irreducibles in $\F_p[x]$.
Unfortunately, no deterministic polynomial time algorithm
is known for this task.
The size of the input is $O(d \log p)$ bits,
so ``polynomial time'' means polynomial in $d$ and $\log p$.
Even the $d = 2$ case,
i.e., computing square roots modulo $p$ in time $(\log p)^{O(1)}$,
remains open.

In this paper, we address an averaged (over $p$) variant of this factorisation problem.
Let $f \in \Z[x]$ be a polynomial with \emph{integer} coefficients,
and let $N$ be a large integer.
For any prime $p$, let $\bar f \in \F_p[x]$ denote the reduction
of $f$ modulo $p$.
(We will systematically use this overline notation to indicate reduction
of various objects modulo~$p$,
the choice of $p$ always being clear from context.)
Consider the problem of finding the factorisation of $\bar f$
\emph{for all primes $p < N$}.
The question is whether this can be done more efficiently than
by simply applying an existing deterministic factorisation algorithm
to $\bar f$ for each prime separately.
The following theorem, which is the main result of the paper,
gives a partially affirmative answer. Here and throughout, any constants implicit in big-$O$ notation are absolute.
\begin{theorem}\label{t:main}
   Let $f \in \Z[x]$ be a monic irreducible polynomial of degree $d \geq 2$
   and with coefficients having absolute value at most $H\geq 2$.
   Let $L$ be the splitting field of $f$ over~$\Q$
   and let $m \coloneqq [L:\Q]$.
   Then we may deterministically compute the factorisation into irreducibles
   of $\bar f \in \F_p[x]$ for all primes $p < N$ in time
   \[ \pi(N) \cdot (m\log H)^{O(1)} \log^5 N, \]
   where $\pi(N)$ denotes the number of primes $p<N$.
\end{theorem}
Since deterministic polynomial-time factorisation over $\Z$ is known (see Appendix~\ref{s:fact-in-z}), one may remove the adjective ``irreducible''  from the first line of the above. Furthermore, if $f$ has leading coefficient $a_d \ne 1$, then replacing $f(x)$ with $a_d^{d-1} f(x/a_d)$ and dealing with those $p$ dividing $a_d$ separately, one could also remove the adjective ``monic''.

Note that the size of the input is $O(d \log H + \log N)$ bits.
If $f$ is \textit{Galois} over $\Q$ in that $L$ is generated over $\Q$ by a single root of $f$ (so $m = d$),
or more generally, if the splitting field has degree $m = d^{O(1)}$,
then the running time per prime is $(d \log H)^{O(1)} \log^5 N$,
which is fully polynomial in the input size.
We also remark that the exponents of $m$ and $\log H$ in Theorem \ref{t:main}
could be worked out more explicitly if desired.

Unfortunately, for general $f$,
the degree $m = [L:\Q]$ could be as large as $d!$, slightly worse than exponential in $d$.
Improving Theorem \ref{t:main}
to achieve a complexity bound depending polynomially on $d$ remains open.

As far as we are aware,
no-one has previously considered the deterministic complexity of
polynomial factorisation when amortised over primes in this way. Of course, for practical computations we always have recourse to the much faster randomised factorisation algorithms,
so the results of the present paper should be viewed as purely theoretical.

\subsection*{Comparison to other factoring algorithms}

The literature on deterministic factorisation is vast;
for a detailed bibliography we refer the reader
to the surveys mentioned earlier. Progress on this problem has been relatively sparse over the past couple of decades (with a couple of exceptions, some noted below), and the aforementioned surveys remain pertinent.
In this section we compare Theorem~\ref{t:main} to a few key results.

For factoring an arbitrary degree $d$ polynomial in $\F_p[x]$,
the best unconditional algorithms known have complexity $(dp)^{O(1)}$,
i.e., polynomial in $d$ but fully exponential in $\log p$.
The dependence on $p$ was initially $p^{1+o(1)}$ \cite{Ber67},
and this was reduced to $p^{1/2+o(1)}$ by Shoup \cite{Sho90}.
For subsequent improvements see \cite[\S9]{vzGSho92}
and \cite[Thm.\,1.1]{Shp99}.
It is not known how to replace the $p^{1/2+o(1)}$ term
by $p^{1/2-\eps}$ for any $\eps > 0$.

Shoup also proved that his algorithm runs in polynomial time ``on average'',
in the sense that if a polynomial is selected uniformly at random from all
polynomials of degree $d$ in $\F_p[x]$,
then the expected running time is polynomial in $d$ and $\log p$
\cite[\S4]{Sho90}.
This averaging is, of course, in a different sense from Theorem~\ref{t:main}.

Schoof showed that for $a \in \Z$,
one can find the square root of $a$ modulo~$p$,
i.e., factor $x^2 - a$ in $\F_p[x]$,
in time $((|a|+1) \log p)^{O(1)}$ \cite{Sch85}.
This is polynomial in $\log p$ but exponential in $\log (|a|+3)$.
Pila \cite{Pil90} extended Schoof's methods to prove that
if $\ell$ is an odd prime and $p \equiv 1 \pmod \ell$,
then one can find the $\ell$-th roots of unity in $\F_p$ (i.e., factor the $\ell$-th cyclotomic polynomial $\Phi_\ell(x) \in \F_p[x]$)
in time $(\log p)^{C_\ell}$ for some $C_\ell > 0$.
Theorem \ref{t:main} solves both problems for all $p < N$
in polynomial time on average:
 $(\log |a| \log N)^{O(1)}$ and $(\ell \log N)^{O(1)}$
per prime respectively.

A number of results on deterministic factoring
have been proved assuming the GRH (Generalised Riemann Hypothesis).
For example:
\begin{itemize}
\item 
R\'onyai \cite{Ron92} showed that for $f \in \Z[x]$,
one can factor $\bar f \in \F_p[x]$ (for a single prime $p$ not dividing the discriminant of $f$)
in time $(m \log H \log p)^{O(1)}$,
where the parameters $m$ and $H$ have the same meaning
as in Theorem~\ref{t:main}.
He thus obtains a similar running time to Theorem~\ref{t:main},
including the polynomial dependence on the degree of the splitting field,
but for a single prime rather than on average over primes.

\item
Evdokimov \cite{Evd94} showed that one can factor $f \in \F_p[x]$ of degree $d$ in time
$(d^{\log d} \log p)^{O(1)}$.
Here the dependence on $\log p$ is polynomial,
and while the dependence on $d$ is not quite polynomial,
it is much closer to polynomial than exponential.
In particular, when the splitting field is large (e.g.~$m = d!$),
this algorithm performs much better with respect to $d$
than R\'onyai's algorithm or our Theorem \ref{t:main}. Recent developments in this line of work may be found in \cite{guo1}, \cite{guo2}. 

\item
There are deterministic algorithms
for factoring $f \in \F_p[x]$ in polynomial time
for ``special'' values of $p$. For example, the case that $p-1$ is sufficiently smooth, i.e., has only small prime divisors,  is addressed by
\cite{MS88} (though it remains unknown whether there are infinitely many such primes). More generally the case when $\Phi_k(p)$ is smooth for some $k$ is addressed in \cite{BvzGL01}.
\end{itemize}
Despite these results, even under GRH there is still no method known for
factoring an arbitrary $f \in \F_p[x]$ in deterministic polynomial time.

\subsection*{Notation, terminology, conventions.}

Throughout the paper we observe the following conventions.
The $\log(\cdot)$ function will always mean
the logarithm to base 2.
Expressions such as $\log \log N$ appearing in complexity bounds
are always tacitly adjusted to take positive values for all arguments
in the domain. Furthermore, when we give upper bounds of the form $x^{O(1)}$ for some positive $x$, this should be understood to mean $(\max(2,x))^{O(1)}$. By the \emph{size} of an integer $n$ we mean~$|n|$,
whereas the \emph{bit size} means the number of bits
in its binary representation.
Given a rational number $c = a/b$ written in reduced form,
we will say that its \textit{height} is the value $\max(|a|,|b|)$.
Finally, we will use Vinogradov notation in addition to big-$O$ notation:
$f\ll g$ is equivalent to $f = O(g)$.

Statements in this paper about algorithm runtimes may be understood to occur in the RAM model. Although it is likely that the runtime of algorithms throughout this paper are dominated by the cost of arithmetic and could be stated in the multitape Turing model \cite[Ch.\,2]{Pap94}, we do not conduct this analysis throughout for the sake of simplicity. An exception is Section~\ref{s:rootfinding}, where we operate in the multitape Turing model and track the cost of data access (which turns out to be negligible), since the algorithms there are slightly ``lower level'' and may be of independent interest. We refer the reader to \cite{CR73} for the relationship between complexities in the respective models.

\subsection*{Acknowledgments}
The author is particularly grateful to David Harvey for many helpful conversations over the years, and without whom this document would not exist. The author is also very grateful to Igor Shparlinski for helpful conversations and generous feedback on earlier versions of this document. Work on this project was supported by an AMS-Simons travel grant.

\section{High level sketch and outline}\label{s:outline}

In this section we describe the outline of the paper
and the ideas behind the main algorithms.

The proof of Theorem \ref{t:main} proceeds in three main steps.
First, in Section \ref{s:rootfinding},
we present an algorithm that finds all roots of an integer polynomial
modulo primes in amortised polynomial time (Theorem \ref{t:rootfinding}).
Next, in Section~\ref{s:galois}, we show how the problem of factorising \textit{Galois} integer polynomials
modulo primes may be reduced to the problem of root-finding modulo primes, allowing us to invoke the results of Section \ref{s:rootfinding}. Thus we obtain an algorithm for factorising Galois integer polynomials: Theorem \ref{t:galois-fact}.
Finally, in Section \ref{s:general},
we show how the factorisation of general irreducible integer polynomials may be reduced to that of the Galois case; this culminates in the proof of Theorem \ref{t:main}.

\subsection*{Root finding}

The root finding step relies on an algorithm of Bernstein \cite{Ber00} which,
given a list $\cN$ of positive integers
and a list $\cP$ of primes,
efficiently finds all pairs $(p, n) \in \cP \times \cN$ such that $p \mid n$.
The naive algorithm (testing every pair separately)
has complexity at least $|\cP| \cdot |\cN|$.
By employing fast integer arithmetic and related techniques,
Bernstein instead achieves a running time that is
quasi-linear in the total bit size of $\cN$ and $\cP$.
(The paper \cite{Ber00} does not appear to have been published;
we give a self-contained presentation of the algorithm
in Section~\ref{s:bernstein}.
A more general version, which also handles multiplicities,
appears as \cite[Alg.\,21.2]{Ber05}.) 

Now, given $N \geq 1$ and a polynomial $f \in \Z[x]$,
to find the roots of $f$ modulo all $p < N$,
we proceed by applying Bernstein's algorithm with
$\cN \coloneqq [f(0), \ldots, f(N-1)]$
and taking $\cP$ to be the set of primes $p < N$.
The details are given in Section \ref{s:rootfinding-for-real}.

\subsection*{Factoring polynomials --- the Galois case}

To describe and motivate the second part of the algorithm,
we first describe Berlekamp's 1967 deterministic algorithm \cite{Ber67}.
Indeed, Section \ref{s:galois} may be loosely described as
a globalisation of Berlekamp's algorithm. 

After reducing to the case of a squarefree polynomial
(this is not difficult; see \cite[\S14.6]{vzGG13}),
Berlekamp's algorithm runs as follows.
Suppose that $f = f_1\cdots f_r$ is the factorisation of $f \in \F_p[x]$
into distinct (unknown) irreducibles,
and consider the $\F_p$-algebra
\begin{equation}
\label{eq:berlekamp-iso}
   R \coloneqq \frac{\F_p[x]}{\gen{f}}
   \cong \frac{\F_p[x]}{\gen{f_1}} \oplus \cdots
         \oplus \frac{\F_p[x]}{\gen{f_r}}.
\end{equation}
The aim is to find an element of $R$ that is zero in
at least one component of the right hand side,
and nonzero in some other component.
This element (after lifting to $\F_p[x]$) has nontrivial gcd with $f$,
yielding a nontrivial factorisation of~$f$.
One then recurses on these factors,
until the complete factorisation is found.

To find such an element of $R$,
Berlekamp uses properties of the $p$th power Frobenius map
$\phi \colon R \to R$ given by $\phi(u) \coloneqq u^p$.
Since $\phi$ is a linear map one may use linear algebra to compute a basis $\BB$ for
$B_\phi \coloneqq \ker(\phi - \id)$, the subspace fixed by~$\phi$,
often called the \emph{Berlekamp subalgebra}.
Via the isomorphism \eqref{eq:berlekamp-iso},
this subspace is identified with $\F_p \oplus \cdots \oplus \F_p$,
where each copy of $\F_p$ is the prime subfield of the corresponding
component $\F_p[x]/\gen{f_i}$.
If $f$ is reducible (i.e., $r > 1$),
then $\BB$ must include some $u$
whose image in $\F_p \oplus \cdots \oplus \F_p$
does not have equal values in all components.
For this $u$, there must therefore exist some $a \in \F_p$ such that
$u - a$ is zero in at least one component and nonzero in another.
The algorithm identifies such $u$ and $a$ by
simply computing $\gcd(f, u - a)$
for all $u \in \BB$ and $a \in \F_p$.

Let us now return to the setting of Theorem \ref{t:main}.
Let $f \in \Z[x]$ be an irreducible polynomial of degree $d \geq 2$,
and consider the number field $K \coloneqq \Q(\theta)$
where $\theta$ is a root of $f$.
We assume further that $f$ is Galois,
meaning that all roots of $f$ lie in $K$. Throughout the rest of this discussion we assume familiarity with some requisite basic algebraic number theory, a brief sketch of which may be found in Appendix~\ref{a:alg-nt}.

Let $\delta_f \in \Z$ be the \textit{discriminant} of $f$, assume that $p\nmid \delta_f d$ (the remaining primes can be handled by other methods), and let $\bar f \in \F_p[x]$ be the reduction of $f$ modulo $p$.
The assumption that $p \nmid \delta_fd$ (which we will now cease to reiterate) implies that
$\bar f$ factors into distinct irreducibles in $\F_p[x]$,
say $\bar f = \bar f_1 \cdots \bar f_r$, so we have 
\[
   R_p \coloneqq \frac{\F_p[x]}{\langle \bar f \rangle}
   \cong \frac{\F_p[x]}{\langle \bar f_1 \rangle} \oplus \cdots \oplus \frac{\F_p[x]}{\langle \bar f_r \rangle}.
\]

To ``lift'' Berlekamp's algorithm to characteristic zero,
we need a global analogue of the Frobenius map $\phi \colon R_p \to R_p$. To this end we note the isomorphisms
\begin{equation}
\label{eq:isomorphisms}
   R_p = \frac{\F_p[x]}{\langle \bar f \rangle}
   \cong \frac{\Z[\theta]}{p\Z[\theta]}
   \cong \frac{\OO_K}{p\OO_K},
\end{equation}
where $\OO_K$ is the ring of integers of $K$. The role of $\phi$ \textit{at a particular irreducible factor $f_i$} is then replicated by  (the reduction mod $p$ of) the \textit{Frobenius element} $\sigma=\sigma_i$ of the Galois group $G \coloneqq \Gal (K/\Q)$ for the corresponding prime ideal $\pp_i$ in $p\OO_K$. This descends to an automorphism $\bar \sigma \colon R_p \to R_p$
that fixes $\langle \bar f_i \rangle$ and acts as the $p$th power map
on the component $\F_p[x]/\langle \bar f_i \rangle$. We note that $\phi$ and $\bar \sigma$
are usually \textit{not} the same map on $R_p$,
although they act in the same way on the $i$-th component.

Having replaced $\phi$ by $\bar \sigma$,
we must also replace the subspace $B = \ker(\phi - \id)$
by $B_{\bar \sigma} \coloneqq \ker(\bar \sigma - \id) \subseteq R_p$.
Note that in general $B_{\bar \sigma} \neq B$;
in other words, $B_{\bar \sigma}$ is usually not equal to
$\F_p \oplus \cdots \oplus \F_p$,
so Berlekamp's separation argument must be modified. To this end, letting $\BB_{\bar \sigma}$ be a basis for $\ker(\bar \sigma - \id)$, we prove that for any pair $f_i, f_j$ of distinct irreducible factors, there exists some $u \in \BB_{\bar \sigma}$ and $a \in \F_p$ such that $u - a$ is zero in the $\F_p[x]/\langle \bar f_i \rangle$ component of $R_p$
and nonzero in the $\F_p[x]/\langle \bar f_j \rangle$ component. This allows us to completely recover the factorisation of $\bar f$ in $\F_p[x]$ from the factors $\gcd(\bar f, u - a)$ (for all $\sigma \in G$, $u \in \BB_{\bar\sigma}$ and $a \in \F_p$).

Of course, this strategy is too slow to implement directly because we would need to consider every $a \in \F_p$ for all $p < N$.
Instead, our plan is to use the fast amortised root-finding algorithm
of Section~\ref{s:rootfinding} to quickly locate the relevant
values of $a$ for all primes simultaneously. Consider for each $\sigma \in G$
the subfield $K^\sigma \coloneqq \ker(\sigma - \id)$ of~$K$,
and define $B_\sigma \coloneqq K^\sigma \cap \Z[\theta]$.
The latter is a $\Z$-submodule of $\Z[\theta]$ of rank $[K^\sigma : \Q]$.
Let $\BB_\sigma$ be a $\Z$-basis for $B_\sigma$.
Using the hypothesis that $p \nmid \delta_f d$,
one can prove (Lemma \ref{l:bases-mod-p}) that
if we reduce the elements of $\BB_\sigma$ modulo $p$,
we get a spanning set for $B_{\bar\sigma}$.
So our goal becomes:
for each $\sigma \in G$, each $u \in \BB_\sigma$, and each prime $p < N$,
find all non-trivial gcds of the form $\gcd(\bar f, \bar u - a)$,
where $\bar u \in \F_p[x]/\langle \bar f \rangle$
indicates the reduction of $u$ modulo $p$.

The final ingredient is the observation that we can
detect these non-trivial gcds for many primes simultaneously. For any $u \in \BB_\sigma$,
let $\tilde u \in \Z[x]$ denote the unique polynomial of degree less than $d$
such that $\tilde u(\theta) = u$,
and consider the resultant
$h_u(a) \coloneqq \res_x(f(x), \tilde u(x) - a) \in \Z[a]$.
This polynomial has the property that its roots modulo $p$
 correspond to those $a \in \F_p$ such that
$\gcd(\bar f, \bar u - a)$ is nontrivial.

We thus arrive at the following algorithm,
whose analysis is summarised in Theorem \ref{t:galois-fact}.
First we perform a series of global computations, independent of $p$:
we compute the Galois group $G = \Gal(K/\Q)$,
a basis $\BB_\sigma$ for $B_\sigma$ for each $\sigma \in G$,
and the resultants $h_{\sigma,u} \in \Z[a]$ for each $u \in \BB_\sigma$.
We then apply Theorem \ref{t:rootfinding} to each $h_{\sigma,u}$ to find its
roots $a \in \F_p$ for all $p < N$.
Finally, working now in $\F_p[x]$ for each $p$,
we compute $\gcd(\bar f, \bar u - a)$
for each $\sigma \in G$, each $u \in \BB_\sigma$,
and each of the roots $a \in \F_p$ found previously. This yields enough information to
recover the factorisations of $\bar f$ for all $p < N$.

\subsection*{Factoring polynomials --- the general case}

Finally, we deduce in Section~\ref{s:general} the factorisation of a general integer polynomial $f$ from the factorisation of a minimal polynomial $g$ for a primitive element of its splitting field (which of course does satisfy that it splits in $\Q[x]/\langle g \rangle$, and so we may use the results from Section \ref{s:galois}). The factorisation of $\bar g$ for all $p<N$ is then use to obtain, for each $p$, a factorisation of $f$ in some explicit finite field of characteristic $p$. From here the factorisation of $\bar f$ (i.e., in $\F_p[x]$) for each $p$ may be deduced by computing the Galois orbits of the factors identified in the previous sentence. 

There are additionally a number of auxiliary statements and algorithms which we will need throughout. These are provided in the various appendices.

\section{Root finding}\label{s:rootfinding}

For a polynomial $h \in \Z[y]$ and a prime $p$, define
\begin{equation} \label{eq:Zp}
Z_p(h) \coloneqq \{0 \leq a < p : h(a) = 0 \bmod p\}.
\end{equation}
In this section we bound the cost of computing $Z_p(h)$ for all primes $p < N$. We note that the record for deterministic root finding at a fixed $p$ is $dp^{1/2+o(1)}$; see \cite{BKS15}. The main theorem of this section is the following. 

\begin{theorem}\label{t:rootfinding}
Let $d, H, N \geq 2$ be integers.
Let $h \in \Z[y]$ be a polynomial of degree~$d$
with coefficients of size at most $H$.
Then we may deterministically compute the sets $Z_p(h)$
for all primes $p < N$ in time
\[ \pi(N) \cdot O\left(B \log(NB) \log^3 N + B \log B \log(dH)\log N \right), \]
where $B \coloneqq \log H + d \log N$.
\end{theorem}

The main ingredient in the proof is an algorithm
due to Bernstein \cite{Ber00};
it is presented here as Proposition \ref{p:parTestDiv}.

\begin{remark}
The quantity $B$ arises as a bound for
the bit size of $h(a) \in \Z$ for $a$ in the range $0 \leq a < N$.
\end{remark}

\begin{remark}
If $N$ is sufficiently large compared to $d$ and $H$,
say $N > (dH)^{C_0}$ for fixed $C_0 > 0$,
then the first term in Theorem \ref{t:rootfinding} dominates and the
overall complexity becomes simply $O(d N \log^4 N)$,
i.e., $O(d \log^5 N)$ on average per prime.
(Without this assumption on $N$,
the first term in Theorem \ref{t:rootfinding}
does not provide enough time to solve the problem,
as the algorithm needs time to read the input polynomial
whose bit size is as large as $\Theta(d \log H)$.) We also note explicitly that the second term arises as the cost of computing $h(a)$ for all $a < N$. Depending on the relative sizes of $N, d, H$ (and in particular if $\log H$ is much larger than $d$), this may be done more efficiently by other methods, such as via the Horner scheme and/or observing the linear recurrence relation between the values $h(a)$. Since the main interest of Theorem \ref{t:rootfinding} is in the regime when $N$ is large and the first term dominates, we will not further address improvements on the second term.
\end{remark}

\begin{remark}
Whereas elsewhere in the paper results may be interpreted in the RAM model, for the algorithms in this section (which may be of some independent interest) we will work in the multitape Turing model \cite[Ch.\,2]{Pap94}. Here we must also account for the cost of data access. This turns out to be negligible, and whenever it is not immediately clear that this is the case,
or where the Turing machine implementation is important to this end,
we will provide a separate explanatory remark.
\end{remark}

Throughout this section we write $M(n) \coloneqq C_1 n \log n$,
where $C_1 > 0$ is a constant chosen so that $n$-bit integers
can be multiplied in at most $M(n)$ bit operations \cite{HvdH21}.
We note that for any $n, n' \in \N$, we have 
\begin{equation}\label{eq:super-add}
   M(n+n') \geq M(n) + M(n').
\end{equation}
If $a$ and $b > 0$ are integers with at most $n$ bits,
we may also compute $a \bmod b$,
i.e., the unique remainder $r$ in the interval $0 \leq r < b$,
in time $O(M(n)) = O(n \log n)$ \cite[\S9.1]{vzGG13}.

For any $n \in \Z$, define
\[ \beta(n) \coloneqq
    \begin{cases}
        \lfloor \log |n| \rfloor + 1, & n \ne 0, \\
        1, & n=0.
    \end{cases} \]
Thus $\beta(n)$ is the number of bits in the binary representation of $|n|$.
The amount of space required to store $n$
on the Turing machine tape is $O(\beta(n))$.
Note that for any sequence $[n_0, \ldots, n_{k-1}]$ of nonzero integers
we have
\begin{equation}\label{eq:beta-add}
\beta\big(\prod_i n_i\big) \leq \sum_i \beta(n_i).
\end{equation}
We will use the following observation repeatedly in this section.
\begin{obs}\label{o:levelAdd}
Let $\cN = [n_0, \ldots, n_{k-1}]$ be a list of nonzero integers.
Given a list of indices $0 = i_0 < i_1 < \cdots < i_r = k$,
consider the decomposition $n_0 \cdots n_{k-1} = n'_0 \cdots n'_{r-1}$
where $n'_j \coloneqq n_{i_j} n_{i_j+1} \cdots n_{i_{j+1}-1}$.
Then, invoking \eqref{eq:super-add} and \eqref{eq:beta-add},
\begin{align*}
   M(\beta(n'_0)) + \cdots + M(\beta(n'_{r-1}))
      & \leq M(\beta(n'_0) + \cdots + \beta(n'_{r-1})) \\
      & \leq M(\beta(n_0) + \cdots + \beta(n_{k-1})).
\end{align*}
\end{obs}

\subsection{Bernstein's algorithm}
\label{s:bernstein}

In this section we work extensively with \emph{binary trees}.
If $T$ is such a tree with $k \geq 1$ leaf nodes,
then each non-leaf node has exactly two children,
and we always ensure that $T$ is ``balanced'' in the sense that
its depth is $\log k + O(1)$.

Let $\cA = [a_0, \ldots, a_{k-1}]$ be a list of nonzero integers.
The \emph{product tree} $T(\cA)$ on $\cA$ is a binary tree
defined recursively as follows.
If $k = 1$, then $T(\cA)$ is a singleton node with value $a_0$.
If $k > 1$, then $T(\cA)$ consists of a root node with value
$a_0 \cdots a_{k-1}$,
and left and right children nodes given respectively by
$T(\cA_0)$ and $T(\cA_1)$
where $\cA_0 \coloneqq [a_0, \ldots, a_{\lfloor k/2 \rfloor - 1}]$
and $\cA_1 \coloneqq [a_{\lfloor k/2 \rfloor}, \ldots, a_{k-1}]$.
In particular, the values at the leaf nodes of $T(\cA)$,
reading from left to right, are $a_0, \ldots, a_{k-1}$.

A binary tree is represented on the Turing machine tape as follows.
Suppose that each node has an associated integer value $v$.
For each node, we first store $v$;
then, if this node is a non-leaf node, we recursively store the subtree
rooted at the left child, followed by the subtree rooted at the right child.
Suitable terminating symbols are used to indicate the tree structure.
For example, the product tree on the sequence $[2,3,5,7,11]$
might be represented by the string
\texttt{2310(6(2|3)|385(5|77(7|11))}
(although working of course in binary rather than decimal).

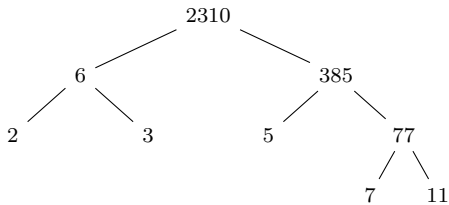
\begin{figure}[h]
\begin{tikzpicture}[
  level distance=8mm,                      
  level 1/.style={sibling distance=34mm}, 
  level 2/.style={sibling distance=18mm}, 
  level 3/.style={sibling distance=9mm},  
  every node/.style={font=\footnotesize},   
  edge from parent/.style={draw, thin},   
  relax/.style={},                         
]
\node {2310}
  child { node {6}
    child { node {2} }
    child { node {3} }
  }
  child { node {385}
    child { node {5} }
    child { node {77}
      child { node {7} }
      child { node {11} }
    }
  };
\end{tikzpicture}
\caption{Product tree on the integers $[2,3,5,7,11]$}
\end{figure}

\begin{prop}[Building a product tree]\label{p:buildTree}
Given a list $\cA\coloneqq[a_0, \ldots, a_{k-1}]$ of nonzero integers,
we may build the product tree $T(\cA)$ in time
\[ O(\ell \log \ell \log k), \qquad
   \ell \coloneqq \sum_{i=0}^{k-1}\beta(a_i). \]
\end{prop}
\begin{proof}
If $k = 1$ then we return a single node with value $a_0$.
Otherwise we recursively compute the product trees $T(\cA_0)$ and $T(\cA_1)$
on $\cA_0 \coloneqq [a_0, \ldots, a_{\lfloor k/2 \rfloor - 1}]$
and $\cA_1 \coloneqq [a_{\lfloor k/2 \rfloor}, \ldots, a_{k-1}]$,
and return a new tree consisting of a root node with children
$T(\cA_0)$ and $T(\cA_1)$ and with value
$a_0 \cdots a_{k-1} = (a_0 \cdots a_{\lfloor k/2 \rfloor - 1})
(a_{\lfloor k/2 \rfloor} \cdots a_{k-1})$,
i.e., the product of the values of the children.
The cost of this multiplication is at most $M(\beta(a_0 \cdots a_{k-1}))$.

By Observation \ref{o:levelAdd},
the total cost of arithmetic at any level of the tree
is at most $M(\ell) = O(\ell \log \ell)$.
The number of levels of the tree is $\log k + O(1)$.
\end{proof}

\begin{remark}\label{r:product-turing}
In the Turing model,
to access the values that need to be multiplied for a given node $t$
(the values of $t$'s children),
we need to traverse a distance that is at most
the size of the tree rooted at $t$.
This distance is $O(\ell_t \log k_t)$,
where $\ell_t$ and $k_t$ are the analogues
of $\ell$ and $k$ for the tree rooted at $t$.
The movement cost is thus bounded above by the
cost of arithmetic at $t$, which is $O(\ell_t \log \ell_t)$.
Similar remarks apply to Proposition \ref{p:testDiv} and
Proposition \ref{p:parTestDiv} below.
\end{remark}

\begin{prop}[Finding small divisors]\label{p:testDiv}
Let $\cP\coloneqq[p_0, \ldots, p_{m-1}]$ be a list of distinct primes.
Given as input the product tree $T(\cP)$
and an integer~$n$ in the interval $0 \leq n < p_0\cdots p_{m-1}$,
we may compute the list $[p \in \cP : p\mid n]$ in time
\[ O(\ell \log \ell \log m), \qquad
   \ell \coloneqq \sum_{i=0}^{m-1} \beta(p_i). \]
\end{prop}
\begin{proof}
The following algorithm is essentially the standard method
for fast multimodular reduction
(see for example \cite[Thm.\,10.24]{vzGG13}).
We work recursively down the tree $T(\cP)$,
appending primes to the output tape as we proceed.
If $m = 1$ (i.e., we are at a leaf node),
then we append $p_0$ to the output if and only if $n = 0$.
Otherwise, if $m > 1$,
let $P_0 \coloneqq p_0 \cdots p_{\lfloor m/2 \rfloor - 1}$
and $P_1 \coloneqq p_{\lfloor m/2 \rfloor} \cdots p_{m-1}$
be the values at the roots of the left and right subtrees $T_0$ and $T_1$.
We compute $n_0 \coloneqq n \bmod P_0$ and $n_1 \coloneqq n \bmod P_1$
at a cost of $O(M(\beta(p_0 \cdots p_{m-1})))$,
and then call the algorithm recursively on $(T_0, n_0)$ and $(T_1, n_1)$.
(Assuming that we always recurse into the left subtree first,
this strategy writes the primes $p_j$ dividing $n$ to the output
in the same order that they appear in the original list $\cP$.)

As in the proof of Proposition \ref{p:buildTree},
the cost of arithmetic at each level is $O(M(\ell)) = O(\ell \log \ell)$,
so the total cost over the whole tree is $O(\ell \log \ell \log m)$.
\end{proof}

\begin{cor}[Finding small divisors of a large integer]\label{c:testDiv}
Given an integer $n \geq 1$ and
a list $\cP\coloneqq[p_0, \ldots, p_{m-1}]$ of distinct primes,
we may compute the list $[p \in \cP : p \mid n]$ in time
\[ O(\beta(n)\log\beta(n) + \ell \log \ell \log m),
   \qquad \ell \coloneqq \sum_{i=0}^{m-1} \beta(p_i). \]
\end{cor}
\begin{proof}
Compute the product tree $T(\cP)$ using Proposition \ref{p:buildTree}.
Reduce $n$ modulo $p_0\cdots p_{m-1}$
and then apply Proposition \ref{p:testDiv}.
\end{proof}

The next result is the core of Bernstein's algorithm.

\begin{prop}[Finding small divisors of many integers]
\label{p:parTestDiv}
Let $\cN\coloneqq[n_0,\ldots,n_{k-1}]$ be a list of nonzero integers,
and let $\cP\coloneqq[p_0,\ldots, p_{m-1}]$ be a list of distinct primes.
For $0 \leq i < k$, let $S_i \coloneqq [p\in \cP: p\mid n_i]$.
Given as input $\cP$ and the product tree $T(\cN)$,
we may compute the list $[S_0, \ldots, S_{k-1}]$ in time
\[ O(\ell \log \ell \log k \log m + \ell' \log \ell' \log m),
\qquad \ell \coloneqq \sum_{i=0}^{k-1} \beta(n_i),
\quad \ell' \coloneqq \sum_{i=0}^{m-1} \beta(p_i). \]
\end{prop}
\begin{proof}
We recurse down $T(\cN)$.
At each node, we first apply Corollary \ref{c:testDiv} to $\cP$
and $n \coloneq n_0 \cdots n_{k-1}$ (the value at the root of $T(\cN)$)
to compute the list of primes
$\cP' \coloneqq [p \in \cP: p \mid n_0 \cdots n_{k-1}]$.
If $k = 1$, we append $\cP'$ to the output
(this is the $S_i$ corresponding to the current leaf node).
Otherwise, if $k > 1$,
we call the algorithm recursively on $(\cP', T_0)$ and $(\cP', T_1)$,
where $T_0 = T([n_0, \ldots, n_{\lfloor k/2 \rfloor - 1}])$
and $T_1 = T([n_{\lfloor k/2 \rfloor}, \ldots, n_{k-1}])$
are the left and right subtrees.

To analyse the complexity, we must bound the total cost
of the invocations of Corollary \ref{c:testDiv}.
For any node $t$,
let us write $\cN_t$, $\cP_t$, $\cP'_t$, $\ell_t$, $\ell'_t$, $m_t$
for the values of the various symbols during the recursive call at $t$.
The cost incurred at $t$ is then
\[ O(\ell_t \log \ell_t + \ell_t' \log \ell_t' \log m_t). \]
In particular, the cost at the root node is
$O(\ell \log \ell + \ell' \log \ell' \log m)$.
We claim that the total cost over the rest of the tree is
$O(\ell \log \ell \log k \log m)$.
To prove this, let us estimate, for each non-leaf node~$t$,
the sum of the costs incurred at its children $t_0$ and $t_1$.
This is given by
\begin{equation}
\label{eq:sum-tj}
   \sum_{j=0,1} O(\ell_{t_j} \log \ell_{t_j} +
      \ell'_{t_j} \log \ell'_{t_j} \log m_{t_j}).
\end{equation}
The key observation is now that
\[ \prod_{p \in \cP_{t_j}} p \mathrel{\big\vert} \prod_{n \in \cN_t} n, \]
which follows from the construction of $\cP_{t_j}\coloneqq [p \in \cP_t : p \mid \prod_{n \in \cN_t} n]$.
Since the $n_i$ are nonzero, we deduce that
$\prod_{p \in \cP_{t_j}} p \leq \prod_{n \in \cN_t} |n|$, and hence that
\begin{multline*}
   \ell'_{t_j} = \sum_{p \in \cP_{t_j}} \beta(p)
   \leq \sum_{p \in \cP_{t_j}} (\log p + 1)
   \leq \sum_{p \in \cP_{t_j}} 2 \log p
   = 2 \log \prod_{p \in \cP_{t_j}} p
   \\
   \leq 2 \log\prod_{n \in \cN_t} |n|
   \leq 2 \sum_{n \in \cN_t} \log |n|
   \leq 2 \sum_{n \in \cN_t} \beta(n) = 2 \ell_t.
\end{multline*}
Clearly $\ell_{t_j} \leq \ell_t$ and $m_{t_j} \leq m$,
so \eqref{eq:sum-tj} becomes simply $O(\ell_t \log \ell_t \log m)$.
Summing over all non-leaf nodes $t$,
and applying Observation \ref{o:levelAdd} in the usual way,
we conclude that the total cost over all non-root nodes is
$O(\ell \log \ell \log k \log m)$.
\end{proof}

\subsection{Proof of Theorem \ref{t:rootfinding}}
\label{s:rootfinding-for-real}

We first give a straightforward estimate for the cost of evaluating
$h \in \Z[y]$ at a single point.
\begin{lemma}\label{l:polyEval}
Let $d, H, N \geq 2$ be integers.
Let $h \in \Z[y]$ be a polynomial of degree~$d$
with coefficients of size at most $H$.
Given an integer $a$ such that $0 \leq a < N$,
we may compute $h(a) \in \Z$ in time
\[ O(B \log B \log(dH)), \qquad B \coloneqq \log H + d \log N. \]
\end{lemma}
\begin{proof}
Let $n$ be the smallest power of two such that $n > d$.
As a preliminary step, we compute the powers $a, a^2, a^4, \ldots, a^n$
by repeated squaring.
By \eqref{eq:super-add} the cost of this step is
\begin{multline*}
   \sum_{i=0}^{\log(n/2)} M(\beta(a^{2^i}))
   \ll \sum_{i=0}^{\log(n/2)} M(2^i \log N)
   \leq M\Bigg(\sum_{i=0}^{\log(n/2)} 2^i\log N \Bigg) \\
   < M(n \log N) \ll M(d \log N) \leq M(B) \ll B \log B.
\end{multline*}

Now write $h(y) = h_0 + h_1 y + \cdots + h_{n-1} y^{n-1}$,
i.e., zero-pad $h$ to length $n$.
Then $h(y) = h^0(y) + y^{n/2} h^1(y)$ where $h^0, h^1 \in \Z[y]$
have degree less than $n/2$.
We compute $h(a)$ by applying this decomposition repeatedly,
i.e., after recursively computing $h^0(a)$ and $h^1(a)$,
we obtain $h(a) = h^0(a) + a^{n/2} h^1(a)$,
using the precomputed value for $a^{n/2}$.
The computation of $h(a)$ thus forms a binary tree of depth $\log n$.

To analyse the complexity, note that
\[ |h(a)| = |h_0 + h_1 a + \cdots + h_{n-1} a^{n-1}|
   \leq H (1 + N + \cdots + N^{n-1}) \leq H N^n, \]
so $\beta(h(a)) \leq \log H + n \log N + O(1)$.
The cost of obtaining $h(a)$ from $h^0(a)$, $h^1(a)$ and $a^{n/2}$ is thus
$O(M(\log H + n \log N))$.
Summing over the tree, the total cost is
\begin{align*}
   O\left( \sum_{i=0}^{\log n}
      2^i M\left(\log H + \frac{n}{2^i}\log N\right) \right)
   & = O\left( \sum_{i=0}^{\log n}
      2^i \left(\log H + \frac{n}{2^i}\log N\right) \log B \right) \\
   & = O((n \log H + n \log n \log N) \log B) \\
   & = O(d \, (\log H + \log d \log N) \log B) \\
   & = O(d \log N \, (\log H + \log d) \log B) \\
   & = O(B \log B \log (dH)).  \qedhere
\end{align*}
\end{proof}

We will of course also need to compute the primes up to $N$.
The following result is proved in \cite{Ser16b}
(see \cite{Ser16a} for an English translation).
\begin{lemma}\label{l:compute-primes}
The list of primes $p < N$ may be computed in time $O(N \log N)$.
\end{lemma}

Now we may prove the main theorem of this section.
\begin{proof}[{Proof of Theorem \ref{t:rootfinding}}]
We begin by invoking Lemma \ref{l:polyEval} to evaluate $h(a)$ for
$a = 0, 1, \ldots, N-1$ in time $O(NB \log B \log(dH))$.
As shown in the proof of Lemma~\ref{l:polyEval},
the bit size of each $h(a)$ is $O(B)$,
so the total bit size of the list $[h(a)]_{a=0}^{N-1}$ is $O(NB)$.
Let $\cN \coloneqq [n_i]_{i=0}^{k-1}$ be the list obtained
from $[h(a)]_{a=0}^{N-1}$ by removing those elements for which $h(a) = 0$.
Since $h$ has at most $d$ roots in $\Z$,
at most $d$ values are removed in this way.

We next compute the product tree $T(\cN)$ in time $O(NB \log(NB) \log N)$
using Proposition \ref{p:buildTree},
and the list $\cP_N$ of primes up to $N$ in time $O(N \log N)$
via Lemma \ref{l:compute-primes}.
The main step is now to use Proposition \ref{p:parTestDiv}
to compute the list $[S_i]_{i=0}^{k-1}$
where $S_i \coloneqq [p\in \cP_N : p \mid n_i]$.
By the Prime Number Theorem we have
$\sum_{p \in \cP_N} \beta(p) \ll \sum_{p < N} \log p = O(N)$,
so the cost of invoking Proposition \ref{p:parTestDiv} is
\[ O(NB \log(NB) \log^2 N + N \log^2 N) = O(NB \log(NB) \log^2 N). \]
Note that each $S_i$ has bit size $O(\beta(n_i))$,
so the total bit size of the list $[S_i]_{i=0}^{k-1}$ is $O(NB)$.
Finally, we construct a list $[S'_a]_{a=0}^{N-1}$
where $S'_a \coloneqq \{p \in \cP_N : p \mid h(a)\}$
by simply copying across the appropriate $S_i$,
and inserting $S'_a = \cP_N$ for those values of $a$ with $h(a) = 0$.
Since there are at most $d$ such values of $a$,
the bit size of the list $[S'_a]_{a=0}^{N-1}$ is $O(NB + dN) = O(NB)$.
The desired sets $Z_p(h)$ are deduced immediately from $[S'_a]_{a=0}^{N-1}$.
\end{proof}

\begin{remark} \label{r:transpose}
Converting from $[S'_a]_{a=0}^{N-1}$ to $[Z_p(h)]_{p < N}$
may be viewed as a ``transpose'' operation.
In the Turing model,
this may be achieved by rewriting $[S'_a]_{a=0}^{N-1}$
as a list of pairs $(a,p)$ ordered lexicographically by $(a,p)$,
sorting the list lexicographically by $(p,a)$,
and then rewriting again as $[Z_p(h)]_{p < N}$.
Using a merge sort algorithm \cite[p. 152]{R90},
the cost of the sorting step is $O(NB \log(NB))$,
which is negligible.
\end{remark}

\section{Factoring polynomials --- the Galois case}\label{s:galois}

Our goal in this section is to prove the following theorem. 

\begin{theorem}\label{t:galois-fact}
Let $d, H, N \geq 2$ be integers.
Let $f \in \Z[x]$ be a monic irreducible polynomial of degree~$d$
with coefficients of size at most $H$.
Assume that $f$ is Galois,~i.e., $f$ splits into linear factors over
$K \coloneqq \Q[x]/\gen{f}$.
Then we may deterministically compute the factorisations of
$\bar f \in \F_p[x]$ for all $p<N$ in time
\[ \pi(N)\cdot O\big(d^3 \log^5 N + (d \log H)^{O(1)} \log^4 N \big). \]
\end{theorem}

\begin{remark}
We note that the first term in the sum in Theorem \ref{t:galois-fact}
dominates if $N$ is large enough compared to $d$ and $H$,
so in this regime the coefficient size $H$
barely has any influence on the complexity.
\end{remark}

\subsection{Setup}

For the rest of Section \ref{s:galois},
we use the following setup and notation.
Let $f \in \Z[x]$ and $K \coloneqq \Q[x]/\gen{f}$
be as in Theorem \ref{t:galois-fact}.
We will write $K = \Q(\theta)$ where $\theta$ is a root of $f$ (so $\theta \coloneqq x + \gen{f} \in K$), and we will assume that $f$ splits over $K$.

Let $\OO_K$ be the ring of integers of~$K$.
We have that $\Z[\theta]$ is a subring of $\OO_K$
with $\rank_\Z \Z[\theta] = [K:\Q] = d$, so $\Z[\theta]$ is an order in $\OO_K$. We note that in general we are unable to compute the ring~$\OO_K$. Finding this ring is about as difficult as factoring the discriminant~$\delta_f$, which we cannot afford, even allowing probabilistic or heuristic algorithms, let alone deterministically. This is why our statements and algorithms work with the subring $\Z[\theta] \subseteq \OO_K$.

Let $G \coloneqq \Gal(K/\Q)$.
Since $f$ is assumed to be Galois,
the extension $K/\Q$ is Galois and $|G| = [K:\Q] = d$.
For $\sigma \in G$ we write $K^\sigma$ for the subfield of $K$
fixed by $\sigma$,
and we define
\begin{equation} \label{eq:B-sigma}
B_\sigma \coloneqq K^\sigma \cap \Z[\theta]
   = \{ \alpha \in \Z[\theta] : \sigma(\alpha) = \alpha \}.
\end{equation}
Then $B_\sigma$ is a subring of $\OO_K$,
and indeed an order in $\OO_{K^\sigma} = \OO_K \cap K^\sigma$.
By Galois theory we have
$[K:K^\sigma] = |\langle \sigma \rangle| = \ord \sigma$,
so the rank of $B_\sigma$ is given by
\[ \rank_\Z B_\sigma = [K^\sigma:\Q] = \frac{d}{\ord \sigma}. \]
We denote by $\BB_\sigma$ a $\Z$-basis for $B_\sigma$;
we assume that one such basis is chosen at the outset for each $\sigma \in G$
(see Proposition \ref{p:global}),
and remains fixed throughout the discussion.

We now consider reductions modulo primes.
For any prime $p$, let $\bar f \in \F_p[x]$ denote
the reduction of $f$ modulo $p$,
and let
\[ R_p \coloneqq \frac{\F_p[x]}{\gen{\bar f}} \]
be the corresponding quotient algebra.
Note that there are natural isomorphisms
\[ R_p \cong \frac{\Z[x]}{\gen{f,p}} \cong \frac{\Z[\theta]}{\gen{p}}. \]

Let $\delta_f \in \Z$ be the discriminant of $f$. We have $\delta_f \neq 0$ since $f$ is irreducible.
We say that a prime $p$ is \emph{good} (for $f$) if it lies in the set
\[ \cP \coloneqq \{ p \text{ prime} : p \nmid \delta_f d \}. \]
If $p$ is good, Proposition \ref{p:algNT} shows that
$\bar f$ factorises as $\bar f = f_1 \cdots f_r$
where the $f_i \in \F_p[x]$ are distinct irreducibles of the same degree
(and where of course $r$ may depend on $p$).
Hence by the Chinese remainder theorem there is an isomorphism
\begin{equation} \label{eq:CRT}
R_p \cong \frac{\F_p[x]}{\gen{f_1}} \oplus \cdots
         \oplus \frac{\F_p[x]}{\gen{f_r}}.
\end{equation}
The components $\F_p[x]/\gen{f_i}$ are finite fields
of the same cardinality $p^{\deg f_i}$.

\subsection{Separating sets}

\begin{definition}
Let $p \in \cP$ and let $\bar f = f_1 \cdots f_r$
be the factorisation of $\bar f \in \F_p[x]$ into distinct irreducibles.
A \emph{separating set} for $\bar f$ is a set $S \subseteq R_p$
such that for any $i, j \in \{1, \ldots, r\}$, $i \neq j$,
there exists some $g \in S$ such that
$g \in \gen{f_i}$ and $g \notin \gen{f_j}$.
\end{definition}

\begin{remark}
In the literature,
separating sets are usually required to be
subsets of the Berlekamp subalgebra
(see for example \cite{Cam83}, \cite{Sho90}).
Our definition is more lenient, allowing any polynomials in $R_p$.
\end{remark}

The aim of this section is to describe a collection of separating sets,
one for each good prime $p$,
that may be computed simultaneously for many $p$
via the root-finding results of Section \ref{s:rootfinding}.
The first step is the following key lemma,
which adapts the separation criterion from Berlekamp's algorithm
to our setting.
The idea of the proof is to replace
the Berlekamp subalgebra $B_\phi \subseteq R_p$
by an analogous subalgebra $B_{\bar\sigma} \subseteq R_p$
for a suitable $\sigma \in G$.

\begin{lemma} \label{l:separation}
Let $p \in \cP$ and let $\bar f = f_1 \cdots f_r$
be the factorisation of $\bar f \in \F_p[x]$ into distinct irreducibles.
Then for any $i, j \in \{1, \ldots, r\}$, $i \neq j$,
there exists some $\sigma \in G$, $u \in \BB_\sigma$ and
$a \in \{0, \ldots, p-1\}$ such that
$\bar u - \bar a \in \langle f_i \rangle$ and
$\bar u - \bar a \notin \langle f_j \rangle$.
\end{lemma}
\begin{proof}
Note that each $\sigma \in G$ can be made to
act on $R_p$ in the following way.
By Proposition \ref{p:algNT},
since $p \in \cP$,
the inclusion $\Z[\theta] \to \OO_K$ induces an isomorphism
$\OO_K/\gen{p} \cong \Z[\theta]/\gen{p}$ ($\cong R_p$).
Each $\sigma \in G$ is an automorphism of $\OO_K$, so gives rise to induced automorphisms
\[ \bar\sigma \colon R_p \to R_p, \qquad \sigma \in G. \]
(Of course $\sigma$ does not necessarily map
$\Z[\theta]$ into $\Z[\theta]$,
so for $p \notin \cP$ there may be no sensible map
$\bar\sigma \colon R_p \to R_p$.) Now choose $\sigma \in G$ to be the \emph{Frobenius element}
associated to the factor $f_i$ for the given index $i$, so (Proposition \ref{p:algNT}) $\bar \sigma$ fixes $\gen{f_i}$ and
\begin{equation} \label{eq:frobenius}
   \bar\sigma(\alpha) = \alpha^p \pmod{f_i}, \qquad \alpha \in R_p.
\end{equation}
Moreover, $\bar\sigma$ permutes the other ideals $\{\gen{f_k}\}_{k \neq i}$.

Next, letting $\beta$ be the unique element of $R_p$ such that
\begin{align*}
\beta & = 1 \pmod{f_i}, \\
\beta & = 0 \pmod{f_k}, \quad k \neq i,
\end{align*}
so by the above discussion $\bar\sigma(\beta) = \beta$. In particular, $\beta$ lies in the subalgebra
\[ B_{\bar\sigma} \coloneqq \ker(\bar\sigma - \id)
   = \{ \alpha \in R_p : \bar\sigma(\alpha) = \alpha \} \subseteq R_p. \]
Now recall that $\BB_\sigma$ is a fixed $\Z$-basis for the
corresponding ``global'' subalgebra (see \eqref{eq:B-sigma})
\[ B_\sigma = \{\alpha \in \Z[\theta] : \sigma(\alpha) = \alpha\}, \]
say $\BB_\sigma = \{u_1, \ldots, u_m\}$.
It is clear that the modulo $p$ reduction map $\Z[\theta] \to R_p$
maps $B_\sigma$ into $B_{\bar\sigma}$.
Using again the assumption that $p \in \cP$,
Lemma \ref{l:bases-mod-p} says
that in fact the reductions $\bar u_1, \ldots, \bar u_m \in B_{\bar\sigma}$
span $B_{\bar\sigma}$.
We may therefore write $\beta$ in the form
\begin{equation} \label{eq:linear-combination}
\beta = \beta_1 \bar u_1 + \cdots + \beta_m \bar u_m,
   \qquad \beta_\ell \in \F_p.
\end{equation}
Since $\bar u_\ell \in B_{\bar\sigma}$,
we have $\bar u_\ell = \bar\sigma(\bar u_\ell) = \bar u_\ell^p \pmod{f_i}$
by \eqref{eq:frobenius}.

Now, for $k \in \{1, \ldots, r\}$,
let $\pi_k \colon R_p \to \F_p[x]/\gen{f_k}$
denote the projection onto the $k$-th component of \eqref{eq:CRT}. 
Taking first the case $k = i$, we see from the above that the elements
\[ a_\ell \coloneqq \pi_i(\bar u_\ell) \in \F_p[x]/\gen{f_i},
   \qquad \ell = 1, \ldots, m \]
satisfy $a_\ell = a_\ell^p$
in the finite field $\F_p[x]/\gen{f_i}$, so  $a_\ell \in \F_p$.
Reducing \eqref{eq:linear-combination} modulo $\gen{f_i}$, we obtain
\[ 1 = \beta_1 a_1 + \cdots + \beta_m a_m. \]

On the other hand, reading \eqref{eq:linear-combination} modulo $f_j$
(i.e., taking $k = j$), we obtain
\[ 0 = \beta_1 \pi_j(\bar u_1) + \cdots + \beta_m \pi_j(\bar u_m). \]
If $\pi_j(\bar u_\ell) = a_\ell$ for all $\ell$,
these two relations lead to the contradiction $0 = 1$;
thus there must exist some $\ell$ such that $\pi_j(\bar u_\ell) \neq a_\ell$.
Taking $u \coloneqq u_\ell$ and
$a \in \{0, \ldots, p-1\}$ to be a lift of $a_\ell$,
we get $\bar u - \bar a = 0 \pmod{f_i}$ and
$\bar u - \bar a \neq 0 \pmod{f_j}$ as desired.
\end{proof}

\begin{remark}
Let us emphasise the point that in the final lines of the previous proof we are only able to make the comparison between $a_\ell$ and $\pi_j(\bar u_\ell)$ because the former lies in the prime subfield of $\F_p[x]/\gen{f_i}$. Fixing an isomorphism $\psi: \F_p[x]/\gen{f_i} \to \F_p[x]/\gen{f_j}$, we make use of the fact that if $a\in R_p$ is constant, then $\psi(\pi_i(a))=\pi_j(a)$ as elements of $\F_p$. 
\end{remark}

To use Lemma \ref{l:separation} in a way that doesn't require iterating over $a\in \{0, \ldots, p-1\}$ for each $p<N$ (which of course would be too expensive),
the key observation is that the values of $a$ produced by the lemma
actually arise as the roots modulo $p$
of a small collection of polynomials in $\Z[x]$,
independently of $p$. 
These polynomials are defined as follows.
For any $u \in \Z[\theta]$,
let $\tilde u(x)$ denote the lift of $u$ to $\Z[x]$. That is, $\tilde u(x)$ is the unique polynomial of degree $< d$
such that $\tilde u(\theta) = u$.
For $\sigma \in G$ and $u \in \BB_\sigma$, define
\[ h_{\sigma,u}(y) \coloneqq \res_x(f(x), \tilde u(x) - y) \in \Z[y], \]
where $\res_x$ is the resultant with respect to $x$
(see Appendix \ref{s:res}).

Then we have:
\begin{prop} \label{p:separatingSets}
Let $p \in \cP$. Then
\[ S_p \coloneqq \bigcup_{\substack{
      \sigma \in G \\ u \in \BB_\sigma, u \notin \Z}}
   \left\{ \bar u - \bar a : a \in Z_p(h_{\sigma,u}) \right\}
   \subseteq R_p \]
is a separating set for $\bar f \in \F_p[x]$.
Moreover we have $|S_p| \leq d^3$.
\end{prop}
(We remind the reader that $Z_p(h_{\sigma,u})$
denotes the set of $a \in \{0, \ldots, p-1\}$ such that
$h_{\sigma,u}(a) = 0 \pmod p$; see \eqref{eq:Zp}.)
\begin{proof}
Let $\bar f = f_1 \cdots f_r$ be the factorisation
into distinct irreducibles.
Given $i, j \in \{1, \ldots, r\}$ with $i \neq j$,
Lemma \ref{l:separation} says that there exists
$\sigma \in G$, $u \in B_\sigma$ and $a \in \{0, \ldots, p-1\}$
such that $\bar u - \bar a \in \gen{f_i}$
and $\bar u - \bar a \notin \gen{f_j}$.

We claim that $u \notin \Z$, i.e., $\deg \tilde u > 0$.
Indeed if $u \in \Z$, then $\bar u - \bar a \in \F_p$, and $\pi_i(\bar u - \bar a) = \pi_j(\bar u - \bar a)$, a contradiction. 

It remains to show that $a \in Z_p(h_{\sigma,u})$,
i.e., that $h_{\sigma,u}(a) = 0 \pmod p$.
By definition
\[ h_{\sigma,u}(y) = \res_x(f(x), \tilde u(x) - y) \in \Z[y]. \]
By Lemma \ref{l:resModP},
to show that $h_{\sigma,u}(a) = 0 \pmod p$
it suffices to show that
\[ \res_x(\bar f(x), \bar{\tilde u}(x) - \bar a) =0\]
in $\F_p$. But this follows from Lemma \ref{l:resultant},
as $\bar f(x)$ and $\bar{\tilde u}(x) - \bar a$
share a nontrivial common factor in $\F_p[x]$, namely $f_i$.

To estimate $|S_p|$,
first observe (by examining the Sylvester matrix) that
$h_{\sigma,u}(y)$ has leading coefficient $\pm1$
and degree at most $\deg f = d$.
Therefore it has at most $d$ roots modulo $p$.
Moreover we have $|G| = d$ and $|\BB_\sigma| = \rank_\Z B_\sigma \leq d$.
It follows immediately that $|S_p| \leq d^3$.
\end{proof}

\subsection{Algorithms}

In this section we describe the algorithms implementing
Theorem \ref{t:galois-fact}. 

We begin by computing some global data depending only on $f$.
\begin{prop}[Compute global data] \label{p:global}
Given integers $d, H \geq 2$ and $f \in \Z[x]$
as in Theorem \ref{t:galois-fact},
in time $(d \log H)^{O(1)}$ we may compute the following:
\begin{enumerate}[label={\upshape(\alph*)}]
\item
The discriminant $\delta_f \in \Z$.
\item
For each $\sigma \in G$,
a matrix $M_\sigma \in M_d(\Z)$
giving the action of $\delta_f \sigma$ on $\Z[\theta]$
with respect to the monomial basis $1, \theta, \ldots, \theta^{d-1}$.
\item
For each $\sigma \in G$,
a basis $\BB_\sigma$ for the subalgebra $B_\sigma$
defined by \eqref{eq:B-sigma}.
\item
The polynomials $h_{\sigma,u} \in \Z[x]$,
for all $\sigma \in G$ and $u \in \BB_\sigma$.
\end{enumerate}
\end{prop}
\begin{proof}
We will defer the details of this proof to the various appendices. Regarding (a), recall that $\delta_f = (-1)^{d(d-1)/2} \res_x(f,f')$, so to compute $\delta_f$ we can compute the determinant of the corresponding Sylvester matrix. It is clear from Appendix \ref{s:res} that this may be done with $(d\log H)^{O(1)}$ bit operations. Part (b) relies on the ability to factorise polynomials over number fields: Theorems \ref{t:factor-over-z}, \ref{thm:numberfield-fact}. The computation of the matrices $M_\sigma$ is then addressed in Corollary \ref{c:galois-galois-f}. The run time bound for part (c) is proven in Lemma \ref{l:compute-zbasis}. Finally, the computation of the polynomials $h_{\sigma, u}=\res_x(f(x),u(x)-\cdot)$ is addressed in Lemma \ref{l:resComplexity}.
\end{proof}

Next we compute some preliminary data depending on $N$.
\begin{lemma} \label{l:reductions}
Given integers $d, H, N \geq 2$ and $f \in \Z[x]$
as in Theorem \ref{t:galois-fact},
in time
\[ (d \log H)^{O(1)} N \log N \]
we may compute the set
\[ \cP_N \coloneqq \{p \in \cP : p < N\} \]
and the reduced polynomials $\bar f \in \F_p[x]$
for all primes $p < N$.
\end{lemma}
\begin{proof}
We may assume that the output of Proposition \ref{p:global} is known.
To compute $\cP_N$,
we first find all the primes up to $N$ via Lemma \ref{l:compute-primes}
in time $O(N \log N)$,
and then simply test whether each $p < N$ divides $\delta_f d$.
The latter takes time
\[ (\log(\delta_f d) + \log p)^{1+\eps}
   < (d \log H)^{O(1)} (\log p)^{1+\eps} \]
for each $p$ (see Lemma \ref{l:disc-bound} for the bound on the size of $\delta_f$). Similarly, we may compute each $\bar f \in \F_p[x]$ in time
\[ d (\log H + \log p)^{1+\eps}
   < (d \log H)^{O(1)} (\log p)^{1+\eps}, \]
so the total over all primes $p < N$ is at most
$(d \log H)^{O(1)} N (\log N)^\eps$.
\end{proof}

Next we use the root-finding results from Section \ref{s:rootfinding}
to compute separating sets for all $p \in \cP_N$.
\begin{prop}[Compute separating sets] \label{p:computeSeparatingSets}
There is an algorithm with the following properties.
Its input consists of integers $d, H, N \geq 2$
and a polynomial $f \in \Z[x]$ as in Theorem \ref{t:galois-fact}.
Its output is the list of separating sets $S_p \subseteq R_p$
(defined in Proposition \ref{p:separatingSets})
for all $p \in \cP_N$.
Assuming that $N > d \log H$, its running time is
\[ O\big( d^3 N \log^4 N + (d \log H)^{O(1)} N \log^3 N\big). \]
\end{prop}
\begin{proof}
We may assume that the output of Proposition \ref{p:global}
and Lemma \ref{l:reductions} is known.
For each $\sigma \in G$ and $u \in \BB_\sigma$ ($u \notin \Z$),
we apply Theorem \ref{t:rootfinding} to compute $Z_p(h_{\sigma,u})$
for all $p < N$.
The complexity is
\[ O\left(NB \log(NB) \log^2 N + NB \log B \log(d'H') \right),
    \qquad B \coloneqq \log H' + d' \log N, \]
where $d' \coloneqq \deg h_{\sigma,u}$
and $H'$ is the size of the coefficients of $h_{\sigma,u}$.
From the definition of the resultant as the determinant of a Sylvester matrix (Appendix \ref{s:res}) we see that $d' \leq d$. Furthermore we have $\log H' < (d \log H)^{O(1)}$ (in fact we have seen that $h_{\sigma,u}$ may be computed with $(d\log H)^{O(1)}$ bit operations, so the bit sizes of its coefficients are certainly of the form $(d \log H)^{O(1)}$). Thus,
\[ B < (d \log H)^{O(1)} + d \log N,
   \qquad \log(d'H') < (d \log H)^{O(1)}. \]
Using the hypothesis $N > d \log H$,
we obtain furthermore $B < N^{O(1)}$ and hence $\log B \ll \log N$.
The cost of each invocation of Theorem \ref{t:rootfinding} is therefore
\begin{align*}
& O\big(NB \log^3 N + (d \log H)^{O(1)} NB \log N\big) \\
& = O\big(N (d \log N + (d \log H)^{O(1)})
      (\log^3 N + (d \log H)^{O(1)} \log N)\big) \\
& = O\big(d N \log^4 N + (d \log H)^{O(1)} N \log^3 N).
\end{align*}
The number of pairs $(\sigma, u)$ is $O(d^2)$,
so the total cost is
\[ O\big(d^3 N \log^4 N + (d \log H)^{O(1)} N \log^3 N). \qedhere \]
\end{proof}

\begin{remark}
In the Turing model, to construct the desired $S_p$ in the above proof,
we must reorganise the sets $Z_p(h_{\sigma,u})$
to be grouped by $p$ rather than by $(\sigma,u)$.
This may be effected by a sorting strategy as in Remark \ref{r:transpose}.
The number of triples $(\sigma,u,p)$ is
$O(d^2 \sum_{p < N} 1) = O(d^2 N / \log N)$,
and the total bit size of the $Z_p(h_{\sigma,u})$ is
$O(d^2 \sum_{p < N} d \log p) = O(d^3 N)$.
Therefore the cost of the sort is
\[ O\big(d^3 N \log(d^2 N/\log N)\big) = O(d^3 N \log N), \]
where we have again used the hypothesis $N > d \log H$.
\end{remark}

The next result shows how to recover the complete factorisation of $\bar f$
from knowledge of the separating set $S_p$.
\begin{lemma}[Recover factorisations] \label{l:refinement}
Let $p \in \cP$.
Given $\bar f \in \F_p[x]$
and the separating set $S_p \subseteq R_p$ for $\bar f$,
we may find the complete factorisation of $\bar f$ in time
\[ d^{O(1)} (\log p)^{1+\eps}. \]
\end{lemma}

\begin{proof}
Suppose that we have found a partial factorisation of $\bar f$,
say $\bar f = f_1 \cdots f_m$,
where the factors $f_i \in \F_p[x]$ have positive degree
but are not necessarily irreducible.
Given $g \in S_p$, we may attempt to refine the factorisation
by computing $s_i \coloneqq \gcd(\tilde g, f_i)$ for each $i$,
where $\tilde g \in \F_p[x]$ denotes a lift of $g$ with $\deg \tilde g < d$.
If $\deg s_i > 0$ and $\deg s_i < f_i$,
then we succeed in improving the factorisation,
replacing $f_i$ by the nontrivial factors $s_i$ and $f/s_i$.
Repeating this process for all $g \in S_p$,
the separation property implies that we will finally arrive at
the complete factorisation of $\bar f$.

To analyse the complexity,
observe that for each $g \in S_p$ we compute at most $d$ GCDs
and quotients of polynomials in $\F_p[x]$, each of degree at most $d$. Thus the cost for each $g$ is certainly $d^{O(1)} (\log p)^{1+\eps}$ (cf. \cite{vzGG13}).
We have $|S_p| \leq d^3$ by Proposition \ref{p:separatingSets},
so the complexity bound follows.
\end{proof}

Finally we may prove the main result of this section.
\begin{proof}[Proof of Theorem \ref{t:galois-fact}]
Recall that we have budgeted 
\[ O\big(d^3 N\log^4 N + (d\log H)^{O(1)} N\log^3 N\big)\]
bit operations to compute the factorisations of $\bar f \in \F_p[x]$ for all $p< N$. We may assume that the output of Lemma \ref{l:reductions} is known so we have the set $\cP_N$ and the reduced polynomials $\bar f \in \F_p[x]$
for all $p < N$.

If $N \leq d \log H$,
we process the primes $p < N$ one at a time
using Shoup's algorithm \cite{Sho90}
to factor each $\bar f$ in time $d^{O(1)} p^{1/2+\eps}$.
The number of primes is at most $N \leq d \log H$,
so the total cost is $(d \log H)^{O(1)} N^{1/2+\eps}$.

Henceforth assume that $N > d \log H$.
Applying Proposition \ref{p:computeSeparatingSets},
we may compute separating sets $S_p \subseteq R_p$
for all $p \in \cP_N$ in time
\[ O\big( d^3 N \log^4 N + (d \log H)^{O(1)} N \log^3 N\big). \]
We then recover the complete factorisations for these primes
via Lemma \ref{l:refinement} in time
\[ \sum_{p \in \cP_N}
   d^{O(1)} (\log p)^{1+\eps}
   < d^{O(1)} N (\log N)^\eps. \]
To handle those primes $p < N$ with $p \mid \delta_f d$ we note that $\log(\delta_f d) < (d \log H)^{O(1)}$ (Lemma \ref{l:disc-bound}),
so the number of such primes is at most $(d \log H)^{O(1)}$.
Using Shoup's algorithm as above,
the cost of factoring $\bar f \in \F_p[x]$ for these primes is
\[ \sum_{\substack{p < N \\ p \mid \delta_f d}} d^{O(1)} p^{1/2+\eps}
      < (d \log H)^{O(1)} N^{1/2+\eps}.   \qedhere \]
\end{proof}

\section{Factoring polynomials --- the general case}\label{s:general}
Let $f\in \Z[x]$ be monic and irreducible of degree $d$. Let $L$ be the splitting field of $f$ and let $m$ be the degree of $L$ over $\Q$. Recall that Theorem \ref{t:main} claims that $f$ may be deterministically factorised modulo $p$ for all $p<N$ with at most $\pi(N)\cdot (m\log H)^{O(1)}\log^5 N$ bit operations, where $O(1)$ conceals an absolute constant. 

As advertised in the introduction, the algorithm for factorising a general $f$ proceeds by computing a minimal polynomial $g$ for a primitive element of $L$, and then factorising $g$ modulo $p$ for all $p<N$ using the algorithm introduced in the previous section. From this we are able to recover the factorisation of $\bar f$ into irreducibles for each $p$. 

We note that after some preliminary global computations (e.g., computing $g$), and after invoking Theorem \ref{t:galois-fact}, all computations are local in the sense that we simply repeat a process at each $p<N$. In particular, there are no speedups obtained via amortisation in this section. We are ready to outline the full algorithm which proves Theorem \ref{t:main}; we will not labour aspects of the algorithm which have been already been addressed in detail in Section \ref{s:galois}. 

\begin{proof}[Proof of Theorem \ref{t:main}]
We begin by computing via Corollary~\ref{c:general-galois-f} a minimal polynomial $g$ for a primitive element $\beta$ for the splitting field $L$ of $f$, and the factorisation of $f$ into linear factors in $L$. This takes $(m\log H)^{O(1)}$ bit operations. Specifically, we obtain polynomials $h_1,\ldots, h_d\in \Q[y]$ such that
\begin{equation}\label{eq:f-linears} 
f(x) = (x-h_1(\beta))\cdots (x-h_d(\beta)) \in L[x].
\end{equation}
We furthermore have that the bit sizes of the heights of the coefficients of $h_i$ are bounded by $(m\log H)^{O(1)}$, and of course that $h_i(\beta)\in \OO_L$ for each $i$.\footnote{We reiterate that we cannot afford to compute $\OO_L$. The polynomials $h_i$, however, are sufficient.} This completes the ``global'' preprocessing.

Next, use Theorem \ref{t:galois-fact} to factorise $g$ modulo $p$ for all $p<N$ with \[\pi(N)\cdot (m\log H)^{O(1)}\log^5N\] bit operations.  For each such $p< N$, choose arbitrarily $\bar g_0$, an irreducible factor of $g$ modulo $p$. In fact, this is all we will need of the output of Theorem \ref{t:galois-fact}. 

As in the Galois case, a small number of primes (those with $p\mid \delta_g$) will need to be dealt with separately, and we will again do so with Shoup's algorithm. Towards noting that the number of such exceptional primes is suitably small, note that since $g$ can be computed with $(m\log H)^{O(1)}$ bit operations, the bit sizes of its coefficients are certainly of size at most $(m\log H)^{O(1)}$. Thus Lemma \ref{l:disc-bound} says that  $\log |\delta_g| \leq (m\log H)^{O(1)}$, and so the number of primes dividing $\delta_g$ is at most $(m\log H)^{O(1)}$. We may therefore argue exactly as in Section \ref{s:galois} to factorise $f$ modulo these primes using Shoup's algorithm within the claimed number of operations. 

For the remainder of the argument, fix a prime $p\nmid \delta_g$. Recall that the assumption $p\nmid \delta_g$ guarantees that $p$ does not divide any of the denominators that appear in the $h_i$ in \eqref{eq:f-linears}. Reducing \eqref{eq:f-linears} modulo $p$ we obtain 
\begin{equation}\label{eq:f-lin-modp} \bar f(x)  = (x-\bar h_1(\beta))\cdots (x-\bar h_d(\beta)) \in \frac{\OO_L}{\gen{p}}[x].
\end{equation}
Recall also that $\OO_L/\gen{p} \cong \Z[\beta]/\gen{p}$ (see Appendix \ref{a:alg-nt}), and that this isomorphism is induced by inclusion $\Z[\beta] \hookrightarrow \OO_L$, so the expression obtained from (\ref{eq:f-linears}) by simply reducing the coefficients of each $h_i$ modulo $p$ gives a factorisation of $\bar f$ in $(\Z[\beta]/\gen{p})[x]$. The cost of this computation is $O((m\log H)^{O(1)}\log^{1+\eps}p)$ bit operations. Finally, observing as we did in Section \ref{s:galois} the isomorphism \[\frac{\Z[\beta]}{\gen{p}} \cong \frac{\F_p[y]}{\langle \overline g \rangle}, \] noting that it is induced by $\beta \mapsto y$, and replacing $\beta$ with $y$ in (\ref{eq:f-lin-modp}), we may ultimately view this as a factorisation of $\bar f$ in $(\F_p[y]/\gen{g})[x]$.

 Next we compute the reduction of each $\bar h_i(y) \in \F_p[y]/\gen{\bar g}$ modulo $\langle \bar g_0 \rangle$ with at most $d^{O(1)} \log^{1+\eps}p$ bit operations; we will denote this reduction by $\tilde h_i(y)$. This gives the following expression for $\bar f(x) \in (\F_p[y]/\gen{\bar g_0})[x]$:
\[ \bar {f}(x) = (x-\tilde h_1(y))\cdots (x-\tilde h_d(y)) \in \frac{\F_p[y]}{\langle \bar g_0 \rangle}[x]\cong \F_{p^{\deg \bar g_0}}[x].\]

Finally, to recover the factorisation of $\bar f$ in $\F_p[x]$, it remains to compute the orbits of $\tilde h_i(y)$ under the $p$th power map. Indeed, $\bar f \in \F_p[x]$, so the $p$th power of any root of $\bar f(x)$ is another root of $\bar f(x)$. Furthermore, iterating this process $\deg \bar g_0$ times and computing the product of the corresponding linear factors yields  an irreducible factor of $\bar f(x)$.

Of course, computing the product of $\tilde h_i(y)$ with itself $p-1$ times is too expensive. Instead, compute $\tilde h_j(y) \coloneqq \tilde h_1(y)^p$ by repeated squaring. This requires $O(\log p)$ multiplications and reductions in $\F_p[y]/\langle \bar g_0 \rangle$, each of which may be conducted with at most $ m^{O(1)} \log^{1+\eps} p$ bit operations. Then compute $\tilde h_j(y)^p$ and iterate until returning to the original factor $\tilde h_1(y)$. 
Note that the size of the orbit is at most $d$, and the product of the corresponding linear factors may be computed with $m^{O(1)} \log^{1+\eps}p$ bit operations. 

We then repeat this process for any $\tilde h_i(y)$ not previously obtained. By iterating in this way, this process recovers the complete factorisation of $\bar f$ into irreducibles in $\F_p[x]$ with at most 
\[m^{O(1)} \log^{2+\eps}p\leq m^{O(1)} \log^{2+\eps}N\]  bit operations. Summing over $p<N$ proves Theorem \ref{t:main}. 
\end{proof}

\appendix

\section{The resultant}\label{s:res}
We recall here some basic facts about the resultant;
for more information, the reader may consult
\cite[Ch.\,6.3]{vzGG13} or \cite[\S3.3.2]{Coh93}.

Let $R$ be a unique factorisation domain.
(In our application, $R$ will be $\Z$, $\Z[y]$ or $\F_p[y]$.)
For $\ell \geq 0$, let $R[x]_\ell$ denote the $R$-module of polynomials
in $R[x]$ of degree less than or equal to $\ell$.

Let $f, g \in R[x]$ be nonzero polynomials
of degree $n, m \geq 0$ respectively,
say $f(x) = \sum_{i=0}^n f_i x^i$ and $g(x) = \sum_{i=0}^m g_i x^i$.
Consider the map
$\varphi_{f,g} \colon R[x]_m \times R[x]_n \to R[x]_{m+n}$
given by $(s, t) \mapsto sf + tg$.
The domain and codomain of $\varphi_{f,g}$
are both free $R$-modules of rank $m+n$.
With respect to the standard monomial bases,
the matrix of $\varphi_{f,g}$ is the
$(m+n) \times (m+n)$ Sylvester matrix
\[ S_{f,g} = \begin{pmatrix}
f_0    &        &        &        & g_0    &        &        &        &        &        \\
f_1    & f_0    &        &        & g_1    & g_0    &        &        &        &        \\
\vdots & f_1    & \ddots &        & \vdots & g_1    & \ddots &        &        &        \\
\vdots & \vdots &        & f_0    & \vdots & \vdots &        & \ddots &        &        \\
\vdots & \vdots &        & f_1    & g_m    & \vdots &        &        & g_0    &        \\
\vdots & \vdots &        & \vdots &        & g_m    &        &        & g_1    & g_0    \\
f_n    & \vdots &        & \vdots &        &        & \ddots &        & \vdots & g_1    \\
       & f_n    &        & \vdots &        &        &        & \ddots & \vdots & \vdots \\
       &        & \ddots & \vdots &        &        &        &        & g_m    & \vdots \\
       &        &        & f_n    &        &        &        &        &        & g_m
\end{pmatrix}. \]
The \emph{resultant} of $f$ and $g$ is defined to be
the determinant of $S_{f,g}$.
It is denoted by $\res_x(f, g) \in R$.
The point of this construction is the following:
\begin{lemma}\label{l:resultant}
The resultant $\res_x(f, g)$ is zero in $R$ if and only if
$\gcd(f, g)$ is nonconstant in $R[x]$.
\end{lemma}
\begin{proof}
See \cite[Cor.\,6.20]{vzGG13}.
\end{proof}

It is almost but not quite true that taking resultants
commutes with reduction modulo $p$.
The next result will suffice for our needs.
Recall that for $h \in \Z[x]$,
we write $\bar h$ for its image in $\F_p[x]$.
\begin{lemma}\label{l:resModP}
Let $f, g \in \Z[x]$ be nonzero, and let $p$ be a prime.
Assume that $f$ is monic and that $\bar g \neq 0$. Then
\[ \res_x(\bar f, \bar g) = 0 \quad \iff \quad
   \overline{\res_x(f, g)} = 0. \]
\end{lemma}
\begin{proof}
This is a special case of \cite[Lem.\,6.25(i)]{vzGG13}.
(The result follows from a straightforward calculation with determinants;
the only slight complication is that the degree of $\bar g$ might be less than
the degree of $g$.)
\end{proof}

Also, let us record the following observation about the runtime of the resultant computation which is needed for Proposition \ref{p:global}. We use notation from there.

\begin{lemma}\label{l:resComplexity}
Computing the resultant $h_{\sigma, u}(a) \coloneqq \res_x(f(x),\tilde u(x)-a)$ as is required in Proposition \ref{p:global} may be done with $(d\log H)^{O(1)}$ bit operations. 
\end{lemma}
\begin{proof}
We give an ad-hoc argument which is slightly silly but suffices for our purposes: we will find $h_{\sigma,u}(a)$ by finding it at many values of $a$ and then interpolating. Note that the Sylvester matrix for $\res_x(f(x),\tilde u(x)-a)$ has dimension at most $2d+1$. Evaluate this matrix at $a=0,1,2,\ldots, d$. Recall that in Proposition \ref{p:global} the elements of $\cB_\sigma$ have already been computed with $(d\log H)^{O(1)}$ bit operations, and so the number of bits required to store each lift $\tilde u$ is certainly at most $(d\log H)^{O(1)}$. Ultimately, the entries of the relevant Sylvester matrices have bit size bounded by $(d\log H)^{O(1)}$. The determinants of these Sylvester matrices may therefore be deterministically computed for $a=0,1,\ldots, d$ with $(d\log H)^{O(1)}$ bit operations (deterministic polynomial-time determinant computations may be done, for example, with the Bareiss algorithm \cite{Bar68}). Then one may recover the coefficients of the polynomial $\res_x(f(x), \tilde u(x)-a)=h_{\sigma,u}(a)$ by computing the inverse of the relevant Vandermonde matrix (again, the Bareiss algorithm may be used here), and computing its product with the vector of polynomial values. 
\end{proof}

\section{Some algebraic number theory}\label{a:alg-nt}
The main purpose of this subsection is to prove Proposition \ref{p:algNT}, which follows from standard results in algebraic number theory. We sketch the necessary background and proof here. The reader may consult, for example, \cite{ST02} or \cite{Neu99} for further relevant background.

We will also need a bound on the size of the discriminant $\delta_f$ which we will note in Lemma \ref{l:disc-bound}. Recall that for $f\in \Z[x]$ monic of degree $d$, the \textit{discriminant} $\delta_f$ of $f$ is defined by $\delta_f \coloneqq (-1)^{d(d-1)/2}\res_x(f(x),f'(x))$. The following lemma is a (crude) consequence of Hadamard's inequality on matrix determinants, which says that the absolute value of a matrix determinant is at most the product of the $\ell^2$ norms of its columns. 

\begin{lemma}\label{l:disc-bound}
Let $f\in \Z[x]$ be monic with degree $d \geq 2$ and coefficients of size at most $H$. Then $\log |\delta_f| \leq (d\log H)^{O(1)}$.
\end{lemma}

Recall the equivalent formulation of the discriminant in terms of the roots of $f$: $\delta_f = \prod_{i<j}(\theta_i - \theta_j)^2$, where  the $\theta_i$ are the (distinct) roots of $f$. This is the incarnation that we will use in the upcoming lemmas. Recall also that the \textit{conductor} of $\Z[\theta]$ in $\OO_K$ is the ideal $\{a \in \OO_K : a\OO_K \subseteq \Z[\theta]\}$. Note that it is an ideal both of $\Z[\theta]$ and $\OO_K$. The following is needed in preparation for proving Proposition \ref{p:algNT}, which is ultimately what is needed for our algorithm in the main text.

\begin{lemma}\label{l:disc}
Let $f\in \Z[x]$ be monic and irreducible. Suppose that $p\in \Z$ is a prime which does not divide the discriminant $\delta_f$ of $f$. Then:
\begin{enumerate}
\item $p\OO_K$ is coprime to the conductor of $\Z[\theta]$ in $\OO_K$,
and
\item $p$ does not ramify in $\OO_K$.
\end{enumerate} 
\end{lemma}
\begin{proof}
This is standard. Recall that we may write $\delta_f = [\OO_K:\Z[\theta]]^2\Delta_K$, where $\Delta_K$ is the discriminant of $K$ (this follows by writing the monomial basis $\{\theta^j\}_{j=0}^{d-1}$ in terms of an integral basis for $\OO_K$ and recalling that $[\OO_K:\Z[\theta]]$ is the magnitude of the determinant of this change of basis matrix).  The first point is then a consequence of the fact that $p\nmid [\OO_K:\Z[\theta]]$. Indeed, note first that $[\OO_K:\Z[\theta]]$ lies in the conductor because $[\OO_K:\Z[\theta]]$ is the cardinality of the quotient $\OO_K/\Z[\theta]$. Then by coprimality in $\Z$ we have that there are integers $a,b$ such that $ap + b[\OO_K:\Z[\theta]]=1$, yielding coprimality of $p\OO_K$ and the conductor. The second point is a consequence of the fact that $p\nmid \Delta_K$ (see, for example \cite[Ch.III, Theorem 2.9]{Neu99}, \cite[Ch.III, Corollary 2.11]{Neu99}).  
\end{proof}

\begin{prop}\label{p:algNT}
Let $f\in \Z[x]$ be monic and irreducible such that $\Q[x]/\gen{f}$ is a Galois field extension of $\Q$ with Galois group $G$. Suppose that $p \in \Z$ is a prime which does not divide the discriminant $\delta_f$ of $f$. Then: 
\begin{enumerate}
\item $\bar f$ factorises into distinct irreducibles $f_i$, each of the same degree,
\item $G$ acts transitively on the set of prime ideals $\gen{f_i}$ in $\F_p[x]/\gen{\bar f}$,
\item for each irreducible factor $f_i$, there exists $\sigma \in G$ which fixes $\gen{f_i}$  and which acts as the Frobenius automorphism (i.e., the $p$th power map) on the extension of fields $(\F_p[x]/\gen{f_i})/\F_p$.
\end{enumerate} 
\end{prop}
\begin{proof}
We claim the following isomorphisms, which are induced by $\theta \mapsto x$:
\[\frac{\OO_K}{\gen{p}} \cong \frac{\Z[\theta]}{\gen{p}} \cong \frac{\F_p[x]}{\gen{\bar f}}.\]
The right hand isomorphism is immediate by noting that both are isomorphic to $\Z[x]/\gen{f,p}$. The left hand isomorphism is induced by the inclusion $\Z[\theta] \hookrightarrow \OO_K$; we just need to observe surjectivity. But Lemma \ref{l:disc} says that $p\OO_K$ is coprime to the conductor of $\Z[\theta]$ in $\OO_K$, so 
\[1 \in p\OO_K + \{a \in \OO_K : a\OO_K \subseteq \Z[\theta]\}.\]
Noting that the conductor is an ideal both of $\OO_K$ and $\Z[\theta]$, we may multiply this equation by any element (say, $b$) of $\OO_K$ to obtain that $b \in p\OO_K + \Z[\theta]$. This proves surjectivity. 

The conclusions of the proposition thus follow from the corresponding claims about prime ideals in $\OO_K/\gen{p}$. The ``distinctness'' claim of (1) follows immediately from Lemma \ref{l:disc} (2). The ``same degree'' claim of (1), and claims (2) and (3) follow from standard facts about how $\Gal(K/\Q)$ acts on $\OO_K/\gen{p}$ (see \cite[Chapter~I,~Section~9]{Neu99} for more details).
\end{proof}

\section{Kernels and bases mod $p$}
Throughout this section, let $p$ be a prime which does not divide $\delta_fd$, where $f$ is irreducible and Galois with Galois group $G\coloneqq \Gal(K/\Q)$. The point of this appendix is twofold: firstly to show in Lemma \ref{l:bases-mod-p} that taking a basis commutes with reduction modulo $p$ under suitable conditions, and secondly to prove Lemma \ref{l:compute-zbasis} which deals with the computation of a $\Z$-basis for $K^\sigma \cap \Z[\theta]$.  In reading this section, the reader may wish to bear in mind the following sequence of isomorphisms addressed in the previous appendix:
\[ \OO_K/\gen{p} \cong \Z[\theta]/\gen{p} \cong \F_p[x]/\gen{\bar f}.\]
Since each of these isomorphisms are suitably trivial (i.e., the first is induced by the inclusion $\Z[\theta] \hookrightarrow \OO_K$, and the second by $\theta \mapsto x$), we often suppress them in our notation, and pass between these algebras without comment. In particular, for $\sigma \in \Gal(K/\Q)$, we may view $\bar \sigma$ as a map on any one of these algebras.

\begin{lemma}\label{l:bases-mod-p}
Let $p$ be a prime with $p \nmid \delta_f d$ and let $\BB_\sigma$ be a $\Z$-basis for $B_\sigma \coloneqq K^\sigma \cap \Z[\theta]$. Then  $\bar \BB_\sigma \coloneqq \BB_\sigma \mod p$ spans $\ker(\bar \sigma-\id)\subset \Z[\theta]/\gen{p}$ over $\F_p$. 
\end{lemma}
\begin{proof}
Write the elements of $\BB_\sigma$ as integer-linear combinations of the elements of a $\Z$-basis $\CC_\sigma$ for $\OO_{K^\sigma}$, where the determinant of the change of basis matrix is of course equal to $[\OO_{K^\sigma}:K^\sigma\cap \Z[\theta]]$. Note next that there is an injection 
\[ \frac{\OO_{K^\sigma}}{\Z[\theta]\cap K^{\sigma}} \hookrightarrow \frac{\OO_K}{\Z[\theta]}\]
which is induced by the inclusion $\OO_{K^\sigma} \hookrightarrow \OO_K$,
so $[\OO_{K^\sigma}:K^\sigma\cap \Z[\theta]]$ divides $[\OO_K:\Z[\theta]]$. In turn, as we saw in the previous appendix, $[\OO_K:\Z[\theta]]$ divides $\delta_f$, so the condition that $p\nmid \delta_f$ gives ultimately that the change of basis matrix from $\CC_\sigma$ to $\BB_\sigma$ is invertible modulo $p$. Clearly both $\bar \BB_\sigma$ and $\bar \CC_\sigma$ are contained in $\ker (\bar \sigma - \id)$, and thus $\bar \BB_\sigma$ spans $\ker(\bar \sigma - \id)$ if and only if $\bar \CC_\sigma$ does. It therefore suffices to show that if  $\CC_\sigma$ is a $\Z$-basis for $\OO_{K^\sigma}$, then its reduction modulo $p$ spans $\ker(\bar \sigma - \id)$ over $\F_p$, which we will do now. 

Let $\bar u \in \ker(\bar \sigma - \id) \subset \OO_K/\gen{p}$, so $\bar \sigma(\bar u) = \bar u$ and we may lift $\bar u$ to $u \in \OO_K$ with $\sigma(u) - u \in p\OO_K$. Next let $\ell\in \Z$ satisfy $\ell = d^{-1} \pmod p$ (recall that $d \coloneqq \deg f = |G|$ and so $\sigma^d=1$), and let 
\[v \coloneqq \ell\left(u + \sigma(u) + \cdots + \sigma^{d-1}(u)\right) \in \OO_K\cap K^\sigma = \OO_{K^\sigma},\]
so $v$ lies in the $\Z$-span of $\CC_\sigma$. Furthermore, we see that $\bar v = d^{-1}(d\bar u) = \bar u$. Therefore $\bar \CC_{\sigma}$ spans $\ker(\bar \sigma - \id)$, completing the proof. 
\end{proof}

Now we deal with computational aspects. 

\begin{lemma}\label{l:compute-zbasis}
Fix $\sigma \in \Gal(K/\Q)$. Given the (integer) matrix $M_\sigma$ for the action of $\delta_f \sigma$ on $K$ with respect to the basis $\{\theta^j\}_{j=0}^{d-1}$, we may deterministically compute a $\Z$-basis for $K^\sigma\cap \Z[\theta]$ with $(d\log H)^{O(1)}$ bit operations. 
\end{lemma}
\begin{proof}
Recall that $\sigma$ acts on $\OO_K$, and as we saw in the proof of Lemma \ref{l:disc}, $[\OO_K:\Z[\theta]]$ divides $\delta_f$, so we do indeed have that $M_\sigma\in M_{d}(\Z)$. Thus, to compute a $\Z$-basis for $K^\sigma \cap \Z[\theta]$, we may compute a $\Z$-basis for $\ker(M_\sigma - I)$. To compute a $\Z$-basis for the kernel of an integer matrix, it suffices to compute its \textit{Hermite normal form} (see, for example, \cite{Coh93}); indeed, after obtaining $H=(M_\sigma-I)Q$, where $H$ is in (columnwise) Hermite normal form and $Q\in\text{GL}_d(\Z)$, the basis for $K^\sigma \cap \Z[\theta]$ may be read off the columns of $Q$ corresponding to the zero columns of $H$.  For an $n\times n$ matrix with entries of size at most $B$, the Hermite normal form computation may be done deterministically with $(n\log B)^{O(1)}$ bit operations using lattice basis reduction (see \cite{vdK00}, \cite{HMM98}, \cite{LLL82}). Furthermore, from Corollary \ref{c:galois-galois-f}, the bit sizes of the integer entries for $M_\sigma$ are at most $(d\log H)^{O(1)}$. Combining these bounds proves the lemma.
\end{proof}

\section{Computing factorisations, Galois groups, and related objects}\label{s:fact-in-z}

Polynomial-time factorisation of (primitive) polynomials over $\Z$ goes back to Lenstra, Lenstra and Lov\' asz. 

\begin{theorem}[Factoring integer polynomials]\cite[Theorem 3.6]{LLL82}\label{t:factor-over-z}
Let $f \in \Z[x]$ be primitive, have degree $d$, and have each of its coefficients of size at most $H$. Then $f$ may be deterministically factorised into irreducibles in $\Z[x]$ with $(d\log H)^{O(1)}$ bit operations. 
\end{theorem}

The above method was extended by A.K. Lenstra \cite{Len83} to obtain the polynomial-time factorisation of polynomials over rings of integers of number fields. Shortly after, S. Landau obtained an analogous theorem by different methods (still using the result of \cite{LLL82} as a black box). After a little simplification from the statement in \cite{Lan85}, one obtains the following.

\begin{theorem}[Factoring polynomials over number fields]\cite[Theorem 2.1]{Lan85} \label{thm:numberfield-fact}
Let $m,n,A,B\geq 2$ be integers. Let $g\in \Z[y]$ be monic, irreducible of degree $m$ with coefficients of absolute value at most $A$. Let $L = \Q[y]/\langle g \rangle$ and let $f \in \OO_L[x]$ be monic with degree $n$ and coefficients of size\footnote{Here we view $f$ as a bivariate polynomial in $x,y$ with coefficients in $\Q$.} at most $B$. Then we may deterministically factorise $f$ in $\OO_L[x]$ with at most $(mn\log (AB))^{O(1)}$ bit operations.
\end{theorem}

This theorem has a number of useful corollaries. Firstly, one may iterate Theorem \ref{thm:numberfield-fact} to obtain the following. The reader may consult \cite[Corollary 6]{Lan85} for details.

\begin{cor}[Computing the splitting field]\label{c:general-galois-f}
Let $H\geq 2$. Let $f\in \Z[x]$ be monic and irreducible with coefficients of size at most $H$. Let $L$ be its splitting field over $\Q$ and let $m\coloneqq [L:\Q]$. Then we may compute (a) a minimal polynomial $g$ for a primitive element $\beta$ for $L$ over $\Q$, and (b) a factorisation of $f$ into linear factors in terms of $\beta$, with at most $(m\log H)^{O(1)}$ bit operations.  In particular, the bit size of the coefficients of  $g$ are $(m\log H)^{O(1)}$, as are the coefficients (with respect to powers of $\beta$) of the linear factors of $f$.
\end{cor}

\begin{cor}[Computing the Galois group]\label{c:galois-galois-f}
Let $d, H\geq 2$. Let $f\in \Z[x]$ be monic and irreducible with degree $d$ and coefficients of size at most $H$. Then we may (a) determine whether $f$ splits into linear factors in $K\coloneqq\Q[x]/\langle f \rangle$, (b) if it does, then obtain the factorisation in $\OO_K[x]$, and (c) if it does, then obtain matrices for the action of the Galois group of $K/\Q$ on $K$ with respect to the monomial basis, all with at most $(d\log H)^{O(1)}$ bit operations. Furthermore, after multiplying by $\delta_f$, the bit size of the (integer) entries of each matrix are at most $(d\log H)^{O(1)}$.
\end{cor}
\begin{proof}
The points (a) and (b) are immediate from Theorem \ref{thm:numberfield-fact}. Suppose now that $f$ splits in $K$ and write $K=\Q(\theta)$. Let $\sigma \in \Gal(K/\Q)$. There is a root $\theta_\sigma$ of $f$ such that $\sigma(\theta) = \theta_\sigma$, and from part (b) we may assume that we have an explicit description $\delta_f\theta_\sigma = \sum_{i=0}^{d-1} c_i \theta^i$, where $c_i\in \Z$. We note that $\log |c_i| \leq (d\log H)^{O(1)}$ since part (b) allows us to compute the $c_i$ with $(d\log H)^{O(1)}$ bit operations, so this is certainly an upper bound for their bit size. 

It remains to compute the matrices for $\delta_f \sigma$ with respect to the basis $\{\theta^j\}_{j=0}^{d-1}$. We may do this, for example, by successively multiplying by $\delta_f \theta_\sigma$ and reducing modulo $f$. The claimed bounds on both the time cost and (as a consequence) the sizes of the matrix entries then follows from being able to deterministically execute this last sentence in polynomial time (cf., e.g., \cite[Chapter 4.6]{Knu69},\cite{BP86}).   
\end{proof}

\bibliographystyle{alpha}
\bibliography{fact.bib}

\end{document}